\documentclass[12pt]{article}
\usepackage[a4paper]{geometry}
\usepackage{latexsym,amssymb,upref,amsmath,amsthm, amsfonts,authblk}
\usepackage{color}
\usepackage{fullpage}
\usepackage{url}

\newtheorem{thm}{Theorem} 
\newtheorem{lemma}[thm]{Lemma}
 
\newtheorem{defn}[thm]{Definition}
\newtheorem{claim}[thm]{Claim}
\newtheorem{prop}[thm]{Proposition}

\newtheorem*{remark}{Remark}



\newcommand{\A}{\mathcal{A}}
\newcommand{\B}{\mathcal{B}}

\newcommand{\h}{\mathcal{H}}

\title{
	Maximum size intersecting families of bounded minimum positive co-degree
}

\begin{document}

\author{
J\'ozsef Balogh\thanks{Department of Mathematics, University of Illinois at Urbana-Champaign, IL, 
USA, and MIPT, Russian Federation. Email: \texttt{jobal@illinois.edu}.
Partially supported by NSF Grant DMS-1764123 and Arnold O. Beckman Research Award (UIUC) Campus Research Board 18132, Simons Fellowship 
and the Langan Scholar Fund (UIUC).}
\qquad
Nathan Lemons\thanks{Theoretical Division, Los Alamos National Laboratory, Email: \texttt{nlemons@lanl.gov}.}
\qquad
Cory Palmer\thanks{Department of Mathematical Sciences, University of Montana. Email: \texttt{cory.palmer@umontana.edu}.
Research supported by a grant from the Simons Foundation \#712036.}

}
\date{
\today}
\maketitle

\begin{abstract}
 Let $\mathcal{H}$ be an $r$-uniform hypergraph. The \emph{minimum positive co-degree} of $\mathcal{H}$, denoted by $\delta_{r-1}^+(\mathcal{H})$, is the minimum $k$ such that if $S$ is an $(r-1)$-set contained in a hyperedge of $\mathcal{H}$, then $S$ is contained in at least $k$ hyperedges of $\mathcal{H}$. For $r\geq k$ fixed and $n$ sufficiently large, we determine the maximum possible size of an intersecting $r$-uniform $n$-vertex hypergraph with minimum positive co-degree $\delta_{r-1}^+(\mathcal{H}) \geq k$ and characterize the unique hypergraph attaining this maximum. This generalizes the Erd\H os-Ko-Rado theorem which corresponds to the case $k=1$.  Our proof is based on the delta-system method. 
\end{abstract}

\section{Introduction}

A hypergraph $\mathcal{H}$ is {\it intersecting} if for every pair of hyperedges $h,h' \in E(\mathcal{H})$ we have $h \cap h' \neq \emptyset$. The celebrated theorem of Erd\H os, Ko and Rado~\cite{EKR} gives that for $n \geq 2r$, the maximum size of an intersecting $r$-uniform $n$-vertex hypergraph is $\binom{n-1}{r-1}$. 
The Erd\H os-Ko-Rado theorem is a cornerstone of extremal combinatorics and has many proofs, extensions and generalizations, see the excellent survey of Frankl and Tokushige~\cite{FT} for a history of extremal problems for intersecting hypergraphs.
We call the unique hypergraph achieving the maximum in the Erd\H os-Ko-Rado theorem a {\it maximal star}, i.e., the hypergraph of all hyperedges containing a given vertex.

The {\it degree} of a set of vertices $S$ in a hypergraph $\mathcal{H}$ is the number of hyperedges containing $S$, i.e., $|\{ h \in E(\mathcal{H}) : S \subseteq h\}|$. 
Denote by $\delta_s(\h)$ the minimum degree of an $s$-element subset of the vertices of $\h$. In this way, $\delta_1(\h)$ is the standard minimum degree of a vertex in $\h$.

Huang and Zhao~\cite{HZ} considered a minimum degree version of the Erd\H os-Ko-Rado theorem. In particular, they proved that for $n \geq 2r+1$, if $\mathcal{H}$ is an intersecting $r$-uniform $n$-vertex hypergraph, then $\mathcal{H}$ has minimum degree  $\delta_1(\h) \leq \binom{n-2}{r-2}$. The Huang-Zhao~\cite{HZ} proof uses the linear algebra method and later a combinatorial proof was given by Frankl and Tokushige~\cite{FT2} for $n \geq 3r$.
Kupavskii~\cite{Ku1} gave an extension of this result and showed that for $t <r$ and $n \geq 2r+3t/(1-t/r)$, every intersecting $r$-uniform $n$-vertex hypergraph $\mathcal{H}$ satisfies $\delta_t(\mathcal{H}) \leq \binom{n-t-1}{r-t-1}$. 

In the more general hypergraph setting, Mubayi and Zhao~\cite{MZ} introduced the notion of co-degree Tur\'an numbers, i.e., the maximum possible value of $\delta_{r-1}(\mathcal{H})$ among all $r$-uniform $n$-vertex hypergraphs $\h$ not containing a specified subhypergraph $\mathcal{F}$. In their paper they give several results that show that the co-degree extremal problem behaves differently from the classical Tur\'an problem.

Motivated by the degree versions of the Erd\H os-Ko-Rado theorem and co-degree Tur\'an numbers we propose studying the following  hypergraph degree condition.

\begin{defn}
    Let $\mathcal{H}$ be a non-empty $r$-uniform hypergraph. The \emph{minimum positive co-degree} of $\mathcal{H}$, denoted $\delta_{r-1}^+(\mathcal{H})$, is the maximum $k$ such that if $S$ is an $(r-1)$-set contained in a hyperedge of $\mathcal{H}$, then $S$ is contained in at least $k$ distinct hyperedges of $\mathcal{H}$. 
\end{defn}

Note that the empty hypergraph is a degenerate case; for simplicity we define its positive co-degree to be zero.

As an example, let us examine hypergraphs that contain no $F_5 = \{abc,abd,cde\}$ to compare the co-degree and positive co-degree settings. Frankl and F\" uredi~\cite{FrFu} (see~\cite{KM} for a strengthening) showed that the complete balanced tripartite $3$-uniform hypergraph has the maximum number of hyperedges among all $3$-uniform $n$-vertex $F_5$-free hypergraphs, for $n$ sufficiently large. This construction has minimum co-degree $0$ and it is easy to see that minimum co-degree at least $2$ guarantees the existence of an $F_5$. On the other hand, the balanced tripartite hypergraph is $F_5$-free and has minimum positive co-degree $n/3$ and it can be shown that minimum positive co-degree strictly greater than $n/3$ implies the existence of an $F_5$.

Note that for ordinary graphs (i.e.\ 2-uniform hypergraphs), the minimum positive co-degree is simply the minimum degree of the non-isolated vertices, which in many extremal problems we may assume is equal to the minimum degree. This suggests positive co-degree as a reasonable notion of ``minimum degree'' in a hypergraph.

The positive co-degree condition has appeared in several other contexts. For example, in~\cite{KMV} the term {\it $d$-full} was used and the authors gave some simple lemmas for hypergraphs with minimum positive co-degree (in the course of proving theorems about extremal numbers for hypergraphs).

In this paper we investigate the maximum size of an intersecting $r$-uniform $n$-vertex hypergraph with positive co-degree at least $k$. As the condition $\delta_{r-1}^+(\mathcal{H}) \geq 1$ is vacuous, the maximum in this case is $\binom{n-1}{r-1}$ as given by the Erd\H os-Ko-Rado theorem.  The unique construction achieving this bound has minimum positive co-degree 1.  On the other hand, as shown in Proposition~\ref{codegree-uniformity-bound}, in an intersecting hypergraph the uniformity gives an upper bound on the minimum positive co-degree, i.e., $r \ge k$. Thus the range of interest for our problem is $2\leq k \leq r$.  In this range we prove that for $n$ sufficiently large the maximum-size intersecting hypergraph with minimum positive co-degree $k$ is given by the following hypergraph.

\begin{defn}
Given integers $r \geq k\geq 1$ an \emph{($r$-uniform) $k$-kernel system} is a hypergraph $\mathcal{H}$ on vertex set $V$ with edges $\mathcal{E}=\{E\in \binom{V}{r}: |E\cap X|\geq k\}$, were $X$ is a distinguished subset of $V$ of size $2k-1$.  The set $X$ is called the \emph{kernel} of $\mathcal{H}$.
\end{defn}

Clearly a $k$-kernel system is intersecting.
Observe that the number of hyperedges in an $r$-uniform $n$-vertex $k$-kernel system $\h$ is 
\[
|E(\h)| = \sum_{i=k}^{\max\{r, 2k-1\}} \binom{2k-1}{i} \binom{n-2k+1}{r-i} \geq \binom{2k-1}{k} \binom{n-2k+1}{r-k}= \Omega(n^{r-k}).
\]

Note that a $1$-kernel system is the hypergraph consisting of all hyperedges containing a fixed vertex $x$, i.e., the maximal hypergraph in the Erd\H os-Ko-Rado theorem. Interestingly, $k$-kernel systems appear as solutions to  maximum degree versions of the Erd\H os-Ko-Rado theorem. Let us give three examples. 

First, a special case of a more general theorem of Frankl~\cite{Fr2} implies that if $\mathcal{H}$ is a maximum-size intersecting $r$-uniform $n$-vertex hypergraph with maximum degree at most $2\binom{n-3}{r-2} + \binom{n-3}{r-3}$, then $\h$ is a $2$-kernel system, provided $n$ is large enough. 

Second, Erd\H os, Rothschild and Szemer\'edi (see~\cite{ERS}) posed the following problem: determine the maximum size of an intersecting $r$-uniform $n$-vertex hypergraph $\mathcal{H}$ such that each vertex contained in at most $c|E(\mathcal{H})|$ hyperedges for $r \geq 3$ and $0 < c <1$. They proved when $c=2/3$ and $n$ large, then a $2$-kernel system is the unique hypergraph attaining this maximum. Frankl~\cite{Fr3} showed that for $2/3 \leq c <1$ and $n$ large enough, $\h$ has no more hyperedges than a $2$-kernel system. 
For $3/5< c < 2/3$ and $n$ large enough, F\"uredi~\cite{Fr3} showed that a $3$-kernel system is one of six non-isomorphic hypergraphs attaining this maximum. 
In the case when $1/2< c \leq 3/5$ and $n$ large enough, Frankl~\cite{Fr3} showed that $\h$ has no more hyperedges than a $3$-kernel system, although the unique hypergraph attaining this maximum is not isomorphic to a $3$-kernel system.

Third, Lemons and Palmer~\cite{LP} proved that $3$-kernel systems are the $r$-uniform $n$-vertex hypergraphs with the largest {\it diversity}, i.e., the difference between the number of hyperedges and the maximum degree
for $n$ large enough (see~\cite{FR,KU} for improvements to the threshold on $n$).

The main result of our paper is as follows:

 \begin{thm}\label{general-k-thm}
    Let $\mathcal{H}$ be an intersecting $r$-uniform $n$-vertex hypergraph with minimum positive co-degree $\delta_{r-1}^+(\mathcal{H}) \geq k$ where $1 \leq k \leq r$. If $\mathcal{H}$ has the maximum number of hyperedges, then $\mathcal{H}$ is a $k$-kernel system for $n$ sufficiently large.
\end{thm}
Theorem~\ref{general-k-thm} holds for $n$ large, roughly double exponential in $r$.
In Section~\ref{small-n-section} we give two results that suggest that Theorem~\ref{general-k-thm} should hold for $n$ at least $cr^{k+2}$, where $c$ is a polynomial in $k$.
It would be interesting to further refine the range of $n$ as a function of $r$ and $k$ where our results hold. Also, we only considered the positive co-degree of $(r-1)$-sets. We can define $\delta^+_s(\mathcal{H})$ to be the minimum $k$ such that if $S$ is an $s$-set contained in a hyperedge of $\mathcal{H}$, then $S$ is contained in at least $k$ distinct hyperedges. There may be interesting problems to be considered under this more general condition. 

\section{Proof of Theorem~\ref{general-k-thm}}

First, let us observe that the uniformity of an intersecting hypergraph is always at least the minimum positive co-degree.

\begin{prop}\label{codegree-uniformity-bound}
    If $\mathcal{H}$ is a non-empty intersecting $r$-uniform $n$-vertex hypergraph with minimum positive co-degree $\delta^+_{r-1}(\mathcal{H}) \geq k$, then $r \geq k$.
\end{prop}

\begin{proof}
Assume, for the sake of a contradiction, that $k > r$.
    Let $h = \{x_1,x_2,\dots, x_r\}$ be a hyperedge of $\mathcal{H}$. The $(r-1)$-set $h \setminus \{x_1\}$ has co-degree at least $k$, so there is a vertex $x_{r+1} \not \in h$ such that $\left(h \setminus \{x_1\}\right) \cup \{x_{r+1}\} $ is a hyperedge of $\mathcal{H}$. Similarly, the $(r-1)$-set $\left(h \setminus \{x_1,x_2\}\right) \cup \{x_{r+1}\}$ has co-degree at least $k$, so there is a vertex $x_{r+2} \not \in h\cup \{x_{r+1}\}$ such that  $\left(h \setminus \{x_1,x_2\}\right) \cup \{x_{r+1},x_{r+2}\}$ is a hyperedge of $\mathcal{H}$. Because $k > r$, we can repeat this process to obtain a hyperedge $\left(h \setminus \{x_1,\dots, x_r\}\right) \cup \{x_{r+1},\dots, x_{2r}\} = \{x_{r+1},\dots, x_{2r}\}$ that is in $\mathcal{H}$. Now we have disjoint hyperedges $h$ and $\{x_{r+1},\dots, x_{2r}\}$ in $\h$ which contradicts the intersecting property.
\end{proof}

An $r$-uniform hypergraph $\mathcal{S}$ is a {\it sunflower} if every pairwise intersection of the hyperedges is the same set $Y$, called the {\it core} of the sunflower. We call the sets $h \setminus Y$ for $h \in E(\mathcal{S})$ the {\it petals} of the sunflower $\mathcal{S}$. Note that the petals are pairwise disjoint. Denote the size of the core of a sunflower $\mathcal{S}$ by $c(\mathcal{S})$.


Let $f(r,p)$ denote the minimum integer such that an $r$-uniform hypergraph with $f(r,p)$ hyperedges contains a sunflower with $p$ petals. 
The Sunflower Lemma of Erd\H os and Rado~\cite{ER} claims that $f(r,p) \leq r! (p-1)^r$.
The determination of $f(r,p)$ is a well-known open problem in combinatorics. A recent breakthrough by Alweiss, Lovett, Wu and Zhang~\cite{ALWZ} gives a bound on $f(r,p)$ of about $(\log r)^{r(1+o(1))}$.

In general we cannot force a sunflower to have a core of a specified size unless we increase the number of hyperedges in the host hypergraph. Mubayi and Zhao (Lemma~6 in~\cite{MZ2}) gives conditions for the existence of a sunflower with a core of bounded size.

\begin{lemma}[Mubayi and Zhao,~\cite{MZ2}]\label{k-core}
Fix integers $r\geq 3$, $k\geq 1$ and $p \geq 1$ and let $C=C(r,p)$ be a large enough constant.
If $\mathcal{G}$ is an $r$-uniform $n$-vertex hypergraph with 
\[
|E(\mathcal{G})| \geq C n^{r-k-1},
\]
then  $\mathcal{G}$ contains a sunflower with $p$ petals and core of size at most $k$.  
\end{lemma}

Observe that Lemma~\ref{k-core} is sharp in the order of magnitude of $n$. Indeed, the $r$-uniform $n$-vertex hypergraph consisting of all hyperedges containing a fixed set $Y$ of $k+1$ vertices contains $\binom{n-k-1}{r-k-1}$ hyperedges, but no sunflower with a core of size at most $k$ as any two hyperedges intersect in at least $k+1$ vertices.
We remark that the problem to determine the best constant $C$ in Lemma~\ref{k-core} is interesting in its own right. In the Appendix at the end of the paper we give a new proof of Lemma~\ref{k-core} that gives an improvement to $C$.

We will need a lower bound on the size of a core of a sunflower in an intersecting hypergraph.

\begin{lemma}\label{sunflower-transversal}
If $\mathcal{S}$ is a sunflower with at least $r+1$ petals in an intersecting $r$-uniform hypergraph $\mathcal{G}$ with $\delta_{r-1}^+(\mathcal{G})\geq k$, then the core $Y$ of $\mathcal{S}$ satisfies $|Y|\geq k$.
\end{lemma}

\begin{proof}
For the sake of contradiction, assume that the core $Y$ of $\mathcal{S}$ is small, i.e., $|Y|<k$.
Observe that $Y$ is a transversal of $\mathcal{G}$, i.e., every hyperedge of $\mathcal{G}$ intersects $Y$. Indeed, 
as the petals of the sunflower $\mathcal{S}$ are pairwise vertex-disjoint, each hyperedge of $\mathcal{G}$ must intersect the core $Y$ in order to intersect each of the at least $r+1$ hyperedges associated with the petals of the sunflower.

Now let $Y'$ be a minimum transversal in $\mathcal{G}$. Thus $|Y'| \leq |Y| < k$ and the minimality of $Y'$ guarantees the existence of a hyperedge $h$ that intersects $Y'$ in exactly one element.
 The $(r-1)$-set $h \setminus Y'$ is contained in at most $k-1$ hyperedges of $\mathcal{G}$; one for each element of $Y'$. This contradicts the positive co-degree condition on $\mathcal{G}$.
\end{proof}


\begin{proof}[Proof of Theorem~\ref{general-k-thm}.]
Let $\mathcal{H}$ be an intersecting $r$-uniform $n$-vertex hypergraph with minimum positive co-degree $\delta_{r-1}^+(\mathcal{H}) \geq k$ where $1 \leq k \leq r$. Moreover, suppose that $\h$ has the maximum number of hyperedges. We will show that $\h$ is a $k$-kernel system for $n$ sufficiently large.

We have observed that a $k$-kernel system has minimum positive co-degree at least $k$, so we may assume that
\[
|E(\mathcal{H})| \geq \binom{2k-1}{k} \binom{n-2k+1}{r-k} = \Omega(n^{r-k}).
\]
Therefore, for $n$ large enough, Lemmas~\ref{k-core} and~\ref{sunflower-transversal} guarantees the existence of a sunflower $\mathcal{S}$ with 
$p=(r+1)  r^{k-1}$ petals and core of size $k$. 
Denote the core of $\mathcal{S}$ by $Y=\{y_1,y_2,\dots, y_k\}$.

Note that in order to apply Lemma~\ref{k-core} we need that the following inequality is satisfied:
\[
\binom{2k-1}{k}\binom{n-2k+1}{r-k} \geq Cn^{r-k-1},
\]
where $C=C(r,p)$ is the constant from Lemma~\ref{k-core}. This is satisfied when
\[
n \geq  \frac{(2r-2k)^{r-k}}{\binom{2k-1}{k}} C. 
\]
The value $C = (pr2^r)^{2^r}$ given in \cite{MZ2} follows from a 
theorem of F\"uredi \cite{Fu3}. 


\begin{claim}\label{sunflower-pair}
There is a set of vertices $Z = \{z_1,z_2,\dots, z_{k-1}\}$ such that $Z \cap Y = \emptyset$ and $Z \cup \{y_k\}$ is the core of a sunflower with $r+1$ petals.
\end{claim}

\begin{proof}
We will prove the following stronger claim: For $0 \leq i \leq k-1$, there is a set of vertices $Z_i=\{z_1,z_2,\dots, z_{i}\}$ such that $Y \cap Z_i = \emptyset$ and $Z_i \cup \{y_k,y_{k-1},\dots, y_{i+1}\}$ is the core of a sunflower $\mathcal{S}_i$ with $(r+1)  r^{k-1-i}$ petals. The claim follows from the case $i=k-1$.

We proceed by induction on $i$. The base case $i =0$ is immediate as $Z_0=\emptyset$ and $\mathcal{S}_0 = \mathcal{S}$ is a sunflower with core $Z_0 \cup \{y_k,y_{k-1},\dots, y_1\} = Y$ with $(r+1)  r^{k-1}$ petals. 
Now suppose $i>0$ and the statement holds for $i-1$. Let $\mathcal{S}_{i-1}$ be a sunflower given by the inductive hypothesis.

For each petal $P$ in $\mathcal{S}_{i-1}$ consider the $(r-1)$-set $P\cup Z_{i-1} \cup \{y_k,\dots, y_{i+1}\} = P\cup Z_{i-1} \cup \{y_k,\dots, y_i\} \setminus \{y_{i}\}$. By the positive co-degree condition on $\mathcal{H}$, the set $P\cup Z_{i-1} \cup \{y_k,\dots, y_{i+1}\}$ is contained in $k$ hyperedges of $\mathcal{H}$. Therefore, as $i \leq k-1$, there is a vertex $x(P)$ such that $x(P) \not \in \{y_1,y_2,\dots, y_i\}$ and $\{x(P)\} \cup P\cup Z_{i-1} \cup \{y_k,\dots, y_{i+1}\}$ is a hyperedge of $\mathcal{H}$.

Now suppose there are distinct vertices $x_1,x_2,\dots, x_{r+1}$ among the vertices in $\{x(P) : P \text{ is a petal in } \mathcal{S}\}$. Let $P_1,P_2,\dots, P_{r+1}$ be the petals corresponding to these vertices, i.e., $\{x_j\} \cup P_j \cup Z_{i-1} \cup \{y_k,\dots, y_{i+1}\} \in E(\mathcal{H})$ for $j=1,2,\dots, r+1$. Then $Z_{i-1} \cup \{y_k,\dots, y_{i+1}\}$ is the core of size $k-1$ of a sunflower with petals $P_j \cup \{x_j\}$ for $j=1,2,\dots, r+1$ in $\mathcal{H}$. This contradicts Lemma~\ref{sunflower-transversal}. Therefore, there are at most $r$ distinct vertices among the vertices in $\{x(P) : P \text{ is a petal in } \mathcal{S}_{i-1}\}$. This implies that there is a vertex $x$ that is the vertex $x(P)$ for at least $\frac{1}{r}|E(\mathcal{S}_{i-1})| \geq (r+1)  r^{k-2-(i-1)}$ petals $P$ in $\mathcal{S}_{i-1}$. Put $z_i=x$ and $Z_i = \{z_1,z_2,\dots, z_i\}$ and let $\mathcal{S}_i$ be the sunflower consisting of $(r+1)  r^{k-1-i}$ hyperedges of $\mathcal{S}_{i-1}$ containing $x = z_i$. Observe that $Z_i \cup \{y_k,\dots, y_{i+1}\}$ is the core of the sunflower $\mathcal{S}_i$ with $(r+1)  r^{k-1-i}$ petals.
\end{proof}

Let $\mathcal{S}_Z$ be a sunflower with $r+1$ petals and core $Z \cup \{y_k\}$ given by Claim~\ref{sunflower-pair}. 
There are at most $(r+1)(r-k)+(k-1)$ vertices
disjoint from $Y$ spanned by $\mathcal{S}_Z$.
As $\mathcal{S}$ has $(r+1) r^{k-1}$ petals, we may choose $r+1$ petals of $\mathcal{S}$ that are vertex-disjoint from the vertices of $\mathcal{S}_Z$. Call the resulting sunflower $\mathcal{S}_Y$. Note that $\mathcal{S}_Y$ has $r+1$ petals and core $Y$.

\begin{claim}\label{kernel-claim}
For every petal $P$ in $\mathcal{S}_Z$ and every $y \in Y$ we have that $P \cup Z \cup \{y\}$ is a hyperedge in $\mathcal{H}$.
\end{claim}

\begin{proof}
Observe that the $(r-1)$-set $P \cup Z$ is contained in the hyperedge $P \cup Z \cup \{y_k\}$, so by the positive co-degree condition $P \cup Z$ is contained
in $k$ hyperedges of $\mathcal{H}$. Moreover, each of these hyperedges must intersect every hyperedge in the sunflower $\mathcal{S}_Y$. As $\mathcal{S}_Y$ has at least $2$ petals, each of the $k$ hyperedges containing $P \cup Z$ must contain a distinct vertex of $Y$.
\end{proof}

We now continue with a technical claim that will imply the theorem.

\begin{claim}\label{technical-claim} 
For every $k$-set $T \subset Y\cup Z$ we have:
\begin{enumerate}
\item[(1)] $Q \cup T \in E(\mathcal{H})$ for every petal $Q$ of $\mathcal{S}_Y$,

\item[(2)] $((Y \cup Z)\setminus T) \cup \{s\} \cup P  \in E(\mathcal{H})$ for every $s \in T$ and petal $P$ of $\mathcal{S}_Z$.
\end{enumerate}
\end{claim}

\begin{proof}
We proceed by induction on $t=|T \cap Z|$. Note that $t \leq k-1$.
When $t=0$ we have that $T =Y$, then (1) is immediate as $Q \cup Y \in E(\mathcal{S}_Y) \subset \mathcal{H}$ and (2) follows from Claim~\ref{kernel-claim}. 

Let $t>0$ and suppose the statement of the claim holds for all smaller values of $t$. As $0<  t\leq k-1$, there exists a $z \in Z \cap T$ and a $y \in Y\setminus T$. Fix an arbitrary petal $Q$ of $\mathcal{S}_Y$. 
Put $T' =  T \cup \{y\} \setminus \{z\}$ and note that $|T' \cap Z| = t-1$. Therefore, by induction, we have $Q \cup T' \in E(\mathcal{H})$ and $((Y\cup Z) \setminus T') \cup \{s'\} \cup P \in E(\mathcal{H})$ for every $s' \in T'$ and petal $P$ of $\mathcal{S}_Z$.

By the positive co-degree condition, the $(r-1)$-set $Q \cup T' \setminus \{y\}$ is contained in at least $k$ hyperedges. Moreover, $Q \cup T' \setminus \{y\}$ is disjoint from the hyperedges of the form $((Y \cup Z) \setminus T') \cup \{y\} \cup P$ where $P$ is a petal of $\mathcal{S}_Z$.
As $\mathcal{S}_Z$ has $r+1$ petals and $\h$ is intersecting, this implies that
the $k$ hyperedges containing $Q \cup T' \setminus \{y\}$ each intersect the $k$-set $((Y \cup Z) \setminus T) \cup \{y\}$. In particular, $(Q \cup T' \setminus \{y\}) \cup \{z\} = Q \cup T$ is a hyperedge of $\mathcal{H}$. This proves (1). 

In order to prove (2), let us fix an arbitrary petal $P$ of $\mathcal{S}_Z$.
Observe that the $(r-1)$-set 
\[
((Y \cup Z) \setminus T) \cup P =((Y \cup Z) \setminus ( T' \cup \{z\} \setminus \{y\})) \cup P  = ((Y \cup Z) \setminus T') \setminus \{z\} \cup \{y\} \cup P
\]
is contained in 
the hyperedge $(Y \cup Z) \setminus T' \cup \{y\} \cup P \in E(\mathcal{H})$ whose existence is given by the inductive hypothesis on (2) with $y=s' \in T'$. Therefore, the positive co-degree condition guarantees that the $(r-1)$-set 
$((Y \cup Z) \setminus T) \cup P$ is contained in $k$ hyperedges. In order for these hyperedges to intersect the $r+1$ hyperedges $Q \cup T$ for each petal $Q$ of $\mathcal{S}_Y$, we have that each set of the form $((Y \cup Z) \setminus T) \cup \{s\} \cup P$ for $s \in T$ must be a hyperedge of $\mathcal{H}$.
\end{proof}

We are now ready to complete the proof of Theorem~\ref{general-k-thm}. Suppose that there is a hyperedge $h \in E(\mathcal{H})$ such that $|h \cap (Y \cup Z)| \leq k-1$. Then there exists a $k$-set $T \subset Y \cup Z$ such that $T$ is disjoint from $h$. Moreover, as $\mathcal{S}_Y$ has at least $r+1$ petals, there is a petal $Q$ in $\mathcal{S}_Y$ that is disjoint from $h$. By Claim~\ref{technical-claim} we have that $T \cup Q  \in E(\mathcal{H})$ which is disjoint from $h \in E(\mathcal{H})$. This violates the intersecting property of $\mathcal{H}$, a contradiction.

Therefore, every hyperedge $h\in E(\mathcal{H})$ intersects $Y \cup Z$ in at least $k$ vertices. This implies that $\mathcal{H}$ is a subhypergraph of a $k$-kernel system, i.e., as $\mathcal{H}$ is edge-maximal, it is exactly a $k$-kernel system.
\end{proof}

\begin{remark} \normalfont
Observe that the proof of Theorem~\ref{general-k-thm} gives a stability result. In particular, if $\mathcal{H}$ has enough edges to apply Lemma~\ref{k-core}, then we have that $\mathcal{H}$ is a subhypergraph of a $k$-kernel system.
\end{remark}

\section{Improved thresholds on $n$}\label{small-n-section}

We now show that in the case $k\leq 3$, Theorem~\ref{general-k-thm} holds for $n \geq c r^{k+2}$. In Theorem~\ref{general-k-thm} we need $n$ to be at least double exponential in $r$.
Recall that two hypergraphs $\mathcal{A}$ and $\mathcal{B}$ are {\it cross-intersecting} if for every pair of hyperedges $A \in E(\mathcal{A})$ and $B \in E(\mathcal{B})$ we have $A \cap B \neq \emptyset$.
Also, a {\it transversal} for a hypergraph $\mathcal{H}$ is a set of vertices $T$ such that $T \cap h \neq \emptyset$ for every hyperedge $h \in E(\h)$. The {\it transversal number} $\tau(\h)$ is the minimum $t$ such that there is a transversal $T$ of $\h$ of size $t$.

We begin with a simple bound on the size of an intersecting hypergraph $\mathcal{H}$ with transversal number $\tau(\mathcal{H})=t$. Stronger results for $\tau(\mathcal{H})=3$ and $\tau(\mathcal{H})=4$ are given by Frankl~\cite{Fr-tau} and Frankl, Ota and Tokushige~\cite{FOT}, but we include an argument for the sake of  completeness and as our argument holds for all $n$ and $t$.

\begin{lemma}\label{tau-bound}
    Fix $n \geq r \geq t$.
	Let $\h$ be an intersecting $r$-uniform $n$-vertex hypergraph with transversal number $\tau(\h) \geq t$. Then
	\[
	|E(\h)| \leq r^{t} \binom{n-t}{r-t}.	
	\]
\end{lemma}

\begin{proof}
    Let us construct a $t$-uniform hypergraph $\mathcal{T}$ with $|E(\mathcal{T})| \leq r^t$ such that for every $h \in E(\mathcal{H})$ there exists a $h' \in E(\mathcal{T})$ with $h' \subset h$. The existence of $\mathcal{T}$ immediately implies the lemma as $|E(\h)| \leq |E(\mathcal{T})| \binom{n-t}{r-t}$.
    
    We proceed iteratively. First select an arbitrary hyperedge $h_1 \in E(\h)$. For each vertex $v_1 \in h_1$, the set $\{v_1\}$ is not a transversal of $\h$, so there is a hyperedge $h_2 \in E(\h)$ that is disjoint from $\{v_1\}$.
    For each vertex $v_2 \in h_2$, the set $\{v_1,v_2\}$ is not a transversal of $\h$, so there is a hyperedge $h_3 \in E(\h)$ that is disjoint from $\{v_1,v_2\}$. We continue this process to select a set of $t$ distinct vertices $v_1,v_2,\dots, v_t$. Let $\mathcal{T}$ be the collection of all $t$-sets constructed in this way. 
    Note that in each step there are at most $r$ choices for the vertex $v_i$, so $|E(\mathcal{T})| \leq r^t$.
    
    Now it remains to show that for every $h \in E(\mathcal{H})$ there exists an $h' \in E(\mathcal{T})$ with $h' \subset h$. Observe that at each step $i$, our hyperedge $h$ must intersect $h_i$, so there is a choice of vertex in $h_i \cap h$. Therefore, there is at least one $r$-set constructed that is contained in $h$.
\end{proof}

We first consider the case of minimum positive co-degree at least $2$.

\begin{prop}
Fix $r \geq 3$ and let $n \geq \frac{1}{3} r^4$.
	Let $\mathcal{H}$ be an intersecting $r$-uniform $n$-vertex hypergraph with minimum positive co-degree $\delta^+_{r-1}(\mathcal{H}) \geq 2$. If $\mathcal{H}$ has the maximum number of hyperedges, then $\mathcal{H}$ is a $2$-kernel system.  
\end{prop}

\begin{proof}
	We distinguish three cases based on the minimum transversal size $\tau(\mathcal{H})$ of $\mathcal{H}$.
	
	\smallskip
	
	\noindent {\bf Case 1:} $\tau(\mathcal{H}) = 1$. 
	
	\smallskip
	Then there is a vertex $x$ in each hyperedge of $\mathcal{H}$. Fix a hyperedge $h \in E(\mathcal{H})$ and observe that the $(r-1)$-set $h \setminus \{x\}$ is contained in exactly one hyperedge which violates the positive co-degree condition.
	
	\smallskip
	
	\noindent {\bf Case 2:} $\tau(\mathcal{H}) \geq 3$. 
	
	\smallskip
	
	Then Lemma~\ref{tau-bound} gives
	\[
	|E(\mathcal{H})| \leq r^3 \binom{n-3}{r-3}
	\]
	which for $n \geq \frac{1}{3} r^4$ is smaller than $3 \binom{n-3}{r-2}$, a contradiction.
	\smallskip
	
	\noindent {\bf Case 3:} $\tau(\mathcal{H}) = 2$. 
	
	\smallskip
	
	Let $\{x,y\}$ be a minimum transversal of $\mathcal{H}$. Consider the $(r-1)$-uniform hypergraphs $\mathcal{H}_x = \{h \setminus \{x\} : h \in E(\h) \text{ and }  h \cap \{x,y\} = \{x\}\}$ and $\mathcal{H}_y = \{h \setminus \{y\} : h \in E(\h) \text{ and } h \cap \{x,y\} = \{y\}\}$. First observe that this pair of hypergraphs is cross-intersecting as $\mathcal{H}$ is intersecting. Now observe that any hyperedge $h \in E(\mathcal{H}_x)$ is a set of size $r-1$ that is contained in a hyperedge of $\mathcal{H}$. Thus, $h$ has co-degree at least $2$, therefore must be a member of $\mathcal{H}_y$. This implies that $\mathcal{H}_x = \mathcal{H}_y$, therefore $\mathcal{H}_x$ is intersecting. 
	
	Now if $\h_x = \h_y$ is not a maximal star, then by the Erd\H os-Ko-Rado theorem we have
	\[
	|E(\h)| < 2 \binom{n-3}{r-2} + \binom{n-2}{r-2} = 3 \binom{n-3}{r-2} + \binom{n-3}{r-3},
	\]
	i.e., $\h$ has fewer hyperedges than a $2$-kernel system, a contradiction.
	 Therefore, every hyperedge of $\mathcal{H}_x$ contains a fixed vertex $z$. This implies that every hyperedge of $\mathcal{H}$ contains at least two of $\{x,y,z\}$, i.e., maximality implies that $\mathcal{H}$ is a $2$-kernel system.
\end{proof}

We now turn to the case when $k = 3$. 
We will need two lemmas. The first is due to Frankl (Proposition~1.4 in~\cite{Fr2}).

\begin{lemma}[Frankl,~\cite{Fr2}]\label{cross}
	Let $\A$ and $\B$ be cross-intersecting hypergraphs on vertex set $[N]$ such that $\A$ is $a$-uniform and $\B$ is $(a+1)$-uniform and intersecting. If $N > 2a+1$, then
	\[
	|\A|+|\B| \leq \binom{N}{a},
	\]
	with equality if and only if either $\B$ is empty and $\A$ has size $\binom{N}{a}$ or 
	both $\A$ and $\B$ are maximal stars containing the same a fixed vertex $q$.
\end{lemma}	


The next lemma gives the size of a minimum transversal for a hypergraph with minimum co-degree at least $k$.


\begin{lemma}\label{tau-is-k}
Fix  $r \geq 3$, $k \geq 2$ and let $n \geq 2{\binom{2k-1}{k}}^{-1}(r-k)r^{k+1}$.
	Let $\mathcal{H}$ be an intersecting $r$-uniform $n$-vertex hypergraph with minimum positive co-degree $\delta_{r-1}^+(\mathcal{H}) \geq k$. If $\mathcal{H}$ has the maximum number of hyperedges, then $\mathcal{H}$ has  transversal number $\tau(\mathcal{H}) = k$.
\end{lemma}

\begin{proof}
	First suppose that $\tau(\h) < k$. As in the proof of Lemma~\ref{sunflower-transversal}, let $X$ be a minimal transversal for $\h$ and consider a hyperedge $h$ that intersects $X$ in exactly one element. Such a hyperedge exists as otherwise $X$ is not minimal. The $(r-1)$-set $h \setminus X$ is contained in at most $k-1$ hyperedges of $\h$; one for each element of $X$. This contradicts the co-degree condition on $\h$.
	
	Now suppose that  $\tau(\h) > k$.
	Lemma~\ref{tau-bound} gives
	$|E(\h)| \leq  r^{k+1}\binom{n-k-1}{r-k-1}$. On the other hand, our construction has at least $\binom{2k-1}{k} \binom{n-2k+1}{r-k}$ hyperedges. Therefore, for $n \geq 2{\binom{2k-1}{k}}^{-1}(r-k)r^{k+1}$ we have a contradiction, thus, $\tau(\h)=k$.
\end{proof}

Finally, we need a technical definition to construct auxiliary hypergraphs from $\mathcal{H}$.

\begin{defn}
	Let $\h$ be an $r$-uniform hypergraph and let $T$ be a fixed set of vertices in $\h$.
	For a subset $S \subset T$ define
	\[
	\h_S^T = \{h \setminus S : h \in E(\h) \text{ and }  h \cap T = S\},
	\]
	i.e., $\h^T_S$ is the $(r-|S|)$-uniform hypergraph constructed by removing $S$ from each hyperedge of $\h$ that intersects $T$ in exactly $S$.	
\end{defn}

For ease of notation we will often denote $\h_S^T$ by $\h_{x_1x_2\dots x_s}^T$ when $S = \{x_1,x_2,\dots, x_s\}$.

\begin{thm}\label{k=3}
Fix $r \geq 3$ and let $n \geq 2 r^5$.
	Let $\mathcal{H}$ be an intersecting $r$-uniform $n$-vertex hypergraph with minimum positive co-degree $\delta^+_{r-1}(\mathcal{H}) \geq 3$. 
	If $\mathcal{H}$ has the maximum number of hyperedges, then $\mathcal{H}$ is a $3$-kernel system.
\end{thm}


\begin{proof}
	By Lemma~\ref{tau-is-k} we may assume the minimum transversal size of $\h$ is $\tau(\h)=3$.
	Let $X=\{x,y,z\}$ be a minimum transversal of $\h$. 

	 Consider the three $(r-1)$-uniform hypergraphs $\h^X_x$, $\h^X_y$ and $\h^X_z$. First observe that any pair of these hypergraphs is cross-intersecting as $\h$ is intersecting. Now observe that any hyperedge $h \in E(\h^X_x)$ is a set of size $r-1$ that is contained in a hyperedge of $\h$, therefore $h$ has co-degree at least $3$. This implies that $h$ is also a member of $\h^X_y$ and $\h^X_z$. Thus, all three hypergraphs $\h^X_x, \h^X_y, \h^X_z$ are the same. Moreover, this implies that $\h^X_x$ is intersecting.
	
	We distinguish three cases based on $\tau(\h^X_x)$.
	
	\smallskip
	
	\noindent {\bf Case 1:} $\tau(\h^X_x) =1$.
		
	\smallskip
	
	Let $u$ be a transversal of $\h^X_x$. Every hyperedge of $\h^X_x,\h^X_y,\h^X_z$ contains $u$, therefore, every hyperedge of $\h$ contains at least two vertices from $\{x,y,z,u\}$. 
	Put $T = X \cup \{u\} = \{x,y,z,u\}$.
	 
	 \begin{claim}
	 	The six hypergraphs $\h^T_{ab}$ for $a,b \in T =  \{x,y,z,u\}$ are equal.
	 \end{claim}
	
	\begin{proof}
		It is enough to show that $E(\h^T_{ab}) \subseteq E(\h^T_{ac})$ for any three vertices $a,b,c \in T$.
		Let $h \in E(\h^T_{ab})$ and consider the $(r-1)$-set $h \cup \{a\}$. By the co-degree condition on $\h$ we have that $h \cup\{a\}$ is contained in at least three hyperedges. Each of these hyperedges includes at least two vertices from $\{x,y,z,u\}$, so $h \cup \{a\}$ is contained in the hyperedge $h \cup \{a,c\}$, i.e., 
		$h \in E(\h^T_{ac})$.
	\end{proof}
	
	Observe that $\h^T_{xy}$ and $\h^T_{zu}$ are cross-intersecting, which implies that $\h^T_{xy}$ is intersecting. 
	Now if $\h^T_{xy}$ is not a maximal star, then by the Erd\H os-Ko-Rado theorem we have
	\[
	|E(\h)| < 6 \binom{n-5}{r-3} + 4 \binom{n-4}{r-3} + \binom{n-4}{r-4} = 10 \binom{n-5}{r-3} + 5 \binom{n-5}{r-4} + \binom{n-5}{r-5},
	\]
	i.e., $\h$ has fewer hyperedges than a $3$-kernel system, a contradiction.
	 Therefore, every hyperedge of $\mathcal{H}_{xy}$ contains a fixed vertex $v$. 
	 As the six hypergraphs $\h^T_{ab}$ for $a,b \in T =  \{x,y,z,u\}$ are equal, we have that every hyperedge of $\mathcal{H}$ contains at least three of $\{x,y,z,u,v\}$, i.e., maximality implies that $\mathcal{H}$ is a $3$-kernel system.
	
	\smallskip
	
	\noindent {\bf Case 2:} $\tau(\h^X_x) = 2$.
	
	\smallskip
	
	Let $u,v$ be a minimal transversal of $\h^X_x$, i.e., every hyperedge of $\h^X_x$ contains at least one of $u,v$. As $\h^X_x=\h^X_y=\h^X_z$, we have that every hyperedge of $\h$ contains at least two vertices from $T = \{x,y,z,u,v\}$.
	Moreover, $\h^T_{xu} = \h^T_{yu} = \h^T_{zu}$ and $\h^T_{xv} = \h^T_{yv} = \h^T_{zv}$ and each of these $(r-2)$-uniform hypergraphs is non-empty (as otherwise $u,v$ would not be a minimal transversal).
	Note that there is no hyperedge that intersects $T$ in exactly $u$ and $v$, so $\h^T_{uv}$ is empty. For simplicity, we consider the empty hypergraph as intersecting.
	
	\begin{claim}
		The hypergraph $\h^T_{ab}$ is intersecting for every $a,b \in T = \{x,y,z,u,v\}$.
	\end{claim}
	
	\begin{proof}
		Suppose not. Then there are hyperedges $A,B \in E(\h^T_{ab})$ such that $A \cap B = \emptyset$.
		By the co-degree condition, the $(r-1)$-set $A \cup \{a\}$ is contained in at least three hyperedges of $\h$. Since each hyperedge of $\h$ contains at least two elements from $T$, there is a hyperedge $A \cup \{a, c\}$ where $c \in T\setminus \{a,b\}$. Similarly, the $(r-1)$-set $B \cup \{b\}$ is contained in some hyperedge $B \cup \{b , d\}$  where $d \in T \setminus \{a,b,c\}$.
		However, the hyperedges $A \cup \{a,  c\}$ and $B \cup \{b, d\}$ are disjoint which violates the intersecting property of $\h$.
	\end{proof}
	
		Now for any $a,b \in T$ we have  $\h^T_{T\setminus\{a,b\}}$ and $\h^T_{ab}$ are cross-intersecting, $\h^T_{T\setminus\{a,b\}}$ is $(r-3)$-uniform and $\h^T_{ab}$ is $(r-2)$-uniform and intersecting. Therefore, as $n-5 > 2(r-3)+1$, 
		we may apply Lemma~\ref{cross} to get
	\[
	|E(\h^T_{ab})| + |E(\h^T_{T\setminus\{a,b\}})| \leq \binom{n-5}{r-3}.
	\]
Thus
\[
|E(\h)| = \sum_{S\subseteq T} |E(\h^T_S)|  \leq 10 \binom{n-5}{r-3} + 5 \binom{n-5}{r-4} + \binom{n-5}{r-5}.
\]

As $\h$ has the maximum number of hyperedges, we must have equality above.
Therefore, we must have that for every $a,b \in T$, the hypergraphs $\h^T_{T\setminus\{a,b\}}$ and $\h^T_{ab}$ have the form of one of the two extremal constructions in Lemma~\ref{cross}. In particular, $\h^T_{ab}$ is either empty or  a maximal star. As $\h^T_{xu} = \h^T_{yu} = \h^T_{zu}$ and $\h^T_{xv} = \h^T_{yv} = \h^T_{zv}$ are non-empty, each is a maximal star. The hypergraphs $\h^T_{xu}$ and $\h^T_{yv}$ are cross-intersecting which implies that all six of these these maximal stars share the same fixed vertex $q$. Therefore, we can replace minimal transversal $u,v$ of $\h_x^X$ with $q$, a contradiction.


	
	\smallskip

	\noindent {\bf Case 3:} $\tau(\h^X_x) \geq 3$.
	
	\smallskip
	
	Then Lemma~\ref{tau-bound} gives
	\[
	|E(\h^X_x)| \leq (r-1)^3 \binom{(n-1)-3}{(r-1)-3} \leq r^3\binom{n-4}{r-4}.
	\]
	The remaining hyperedges of $\h$ are counted by $\h^X_{xyz}$ and $\h^X_{ab}$ for $a,b \in \{x,y,z\}$. We need a simple claim. Recall that the {\it shadow} of an $r$-uniform hypergraph $\mathcal{G}$ is the collection of all $(r-1)$-sets contained in a hyperedge of $\mathcal{G}$. We denote the shadow of $\mathcal{G}$ by $\Delta(\mathcal{G})$.
	
	\begin{claim}\label{shadow}
		For each hyperedge $h \in E(\h^X_{yz})$ there is some hyperedge $g \in E(\h^X_x)$ that contains $h$. Thus,
		\[
		|E(\h^X_{yz})| \leq |\Delta(\h^X_x)|.
		\]
	\end{claim}

	\begin{proof}
	    Let $h$ be an arbitrary hyperedge of $\h^X_{yz}$.
		Consider the $(r-1)$-set $A = h \cup \{y\}$. The set $A$ has co-degree at least $3$, so it is contained in three hyperedges of $\h$; one such hyperedge is $A \cup \{z\}$, another could be $A \cup \{x\}$, so there exists at least one hyperedge of the form $A \cup \{w\}$ where $w \not \in \{x,y,z\}$. However, $A \cap \{x,y,z\} = \{y\}$, so $(A \cup \{w\}) \setminus \{y\} \in E(\h^X_y) = E(\h^X_x)$.
	\end{proof}
	
	By Claim~\ref{shadow} we have
	\[
	|E(\h^X_{yz})| \leq |\Delta(\h^X_x)| \leq (r-1) |E(\h^X_x)| \leq r^4\binom{n-4}{r-4}.
	\]
	Finally, $|E(\h^X_{xyz})| \leq \binom{n-3}{r-3}$. Thus,
	\[
	|E(\h)| \leq \binom{n-3}{r-3} + 3(r^4+r^3)\binom{n-4}{r-4}
	\]
	which is less than $10 \binom{n-5}{r-3}$ for $n \geq 2r^5$, a contradiction.
\end{proof}

In order to extend the technique used in this section to reprove our theorem for minimum positive co-degree $k \geq 4$ we would need to distinguish additional cases based on the transversal size of $\h_x^X$. Some of these cases can be addressed with Lemmas~\ref{tau-bound} and \ref{cross}, but probably new ideas will be needed.


\section*{Acknowledgements}
The authors would like to thank Dhruv Mubayi for pointing out reference~\cite{MZ2} and Lemma~\ref{k-core}. We also thank the anonymous referees for their careful reading of the manuscript and many helpful comments that improved the presentation.

\section*{Appendix}

We now give an improvement to Lemma~\ref{k-core} which we believe is of independent interest. Recall that $f(r,p)$ is the minimum integer such that an $r$-uniform hypergraph with $f(r,p)$ hyperedges contains a sunflower with $p$ petals.

\begin{lemma}\label{k-core-new}
Fix integers $r\geq 3$, $k \geq 1$ and $p \geq 1$ and let $n$ be large enough.
If $\mathcal{G}$ is an $r$-uniform $n$-vertex hypergraph with 
\[
|E(\mathcal{G})| \geq 2 r^{r-k} f(r,p r^{r-k}) \binom{n-k-1}{r-k-1},
\]
then  $\mathcal{G}$ contains a sunflower with $p$ petals and core of size at most $k$.  
\end{lemma}


This replaces the value of $C=(pr2^r)^{2^r}$ in Lemma~\ref{k-core} with $C=2 r^{r-k} f(r,p r^{r-k})$ which is significantly smaller when using the bound on $f(r,pr^{r-k})$ from \cite{ALWZ}.

\begin{proof}
For the sake of a contradiction, suppose that $\mathcal{G}$ contains no sunflower with $p$ petals and core of size at most $k$.

Iteratively remove from $\mathcal{G}$ a sunflower $\mathcal{S}$ with exactly $p  r^{c(\mathcal{S})-k}$ petals such that at each step we choose a sunflower with minimum available core size $c(\mathcal{S})$. Let $t$ be the number of steps in this sunflower removal procedure. Note that $t$ grows with $n$ as at each step we remove at most $p r^{r-k}$ hyperedges from $\mathcal{G}$ and we only need constant number of hyperedges to guarantee the existence of a sunflower with $p  r^{c(\mathcal{S})-k}$ petals. In particular, we have
\[
t \geq \frac{|E(\mathcal{G})|-f(r,p  r^{r-k})}{p  r^{r-k}} \geq \frac{|E(\mathcal{G})|}{2pr^{r-k}}
\]
for $n$ large enough.

The core of each removed sunflower is of size at least $k+1$ and at most $r-1$. Therefore, there is some integer $s$ such that there are at least $t /r$ cores of size $s$ among the removed sunflowers.
 Some of these cores may be identical. Let us compute the maximum multiplicity of a core $Y$. There are at most $\binom{n-|Y|}{r-|Y|}$ hyperedges containing $Y$ and each removed sunflower with core $Y$ has exactly $p r^{|Y|-k}$ hyperedges. Therefore,
 the maximum multiplicity of a core $Y$ is at most
\[
\frac{1}{p  r^{|Y|-k}}\binom{n-|Y|}{r-|Y|} \leq \frac{1}{p  r} \binom{n-k-1}{r-k-1}
\]
for $n\geq r$.
Therefore, there is a collection of at least 
\[
(t/r)  p  r \binom{n-k-1}{r-k-1}^{-1} \geq p  \frac{|E(\mathcal{G})|}{2pr^{r-k}}  \binom{n-k-1}{r-k-1}^{-1} \geq f(r,p  r^{r-k})
\]
distinct cores of size $s$. Let $Y_1,Y_2,\dots, Y_q$ be these cores and let $\mathcal{S}_i$ be the sunflower with  core $Y_i$ for $i=1,2,\dots,q$. Note that each of these sunflowers has exactly $p  r ^{s-k}$ petals.

Let $t$ be the first step in the sunflower removal procedure in which a sunflower with core of size $s$ is chosen to be removed. This implies that all later cores are of size at least $s$. Now we will show that there is a sunflower $\mathcal{B}$ with core of size less than $s$ and $p r^{c(\mathcal{B})-k}$ petals among the hyperedges in the sunflowers $\mathcal{S}_1,\mathcal{S}_2,\dots, \mathcal{S}_q$. Before removing the sunflower in step $t$, all hyperedges of the sunflowers $\mathcal{S}_1,\mathcal{S}_2,\dots, \mathcal{S}_q$ are still in $\mathcal{H}$. Therefore, the sunflower $\mathcal{B}$ with core of size less than $s$ could be chosen in step $t$, this will contradict the choice of $t$.

We may think of the $s$-sets $Y_1,\dots, Y_q$ as an $s$-uniform hypergraph on the vertex set of $\mathcal{H}$.
As $q \geq f(r,p  r^{r-k})\geq f(s,p  r^{r-k}) \geq f(s,pr^{s-k})$, the $s$-sets $Y_1,\dots, Y_q$ contain an $s$-uniform sunflower $\mathcal{A}$ with $p  r^{s-k}$ petals and core $Y^*$ of size less than $s$. By relabelling, we may suppose that $Y_i$ is a member of $\mathcal{A}$ for $i=1,2,\dots, p  r^{s-r}$. Note that the petals $Y_i \setminus Y^*$ of $\mathcal{A}$ are pairwise disjoint by definition. The sunflower $\mathcal{A}$ is not in the hypergraph $\mathcal{H}$ as it is $s$-uniform. However, each hyperedge of $\mathcal{A}$ is the core of some sunflower $\mathcal{S}_i$ in $\mathcal{H}$. Therefore, we will use the members of $\mathcal{A}$ to identify an $r$-uniform sunflower $\mathcal{B}$ with core $Y^*$ in $\mathcal{H}$. The main idea will be carefully choose a petal from each sunflower $\mathcal{S}_i$ whose core is a member of $\mathcal{A}$. To this end, define $\mathcal{B}$ as follows:

First pick any hyperedge of $\mathcal{S}_1$; denote it by $h_1$. Now suppose we have chosen $\ell$ hyperedges $h_1,h_2,\dots, h_\ell$ that form a sunflower with core $Y^*$. The union of these hyperedges contains $\ell(r-|Y^*|)$ vertices outside of $Y^*$. Therefore, as long as 
\begin{equation}\label{ell-bound}
p  r ^{s-k}> \ell(r-|Y^*|),
\end{equation}
there is a petal $Y_i \setminus Y^*$ of $\mathcal{A}$ that is disjoint from each of the hyperedges $h_1,h_2,\dots, h_\ell$. The corresponding sunflower $\mathcal{S}_i$ with core $Y_i$ has 
\[
p  r ^{s-k} > \ell ({r-|Y^*|})
\]
petals by (\ref{ell-bound}). Therefore, there is a petal $P$ of $\mathcal{S}_i$ that is also disjoint from the hyperedges in $h_1,h_2,\dots, h_\ell$. Let $h_{\ell+1}$ be the hyperedge $P \cup Y_i$. Now we have a sunflower with $\ell+1$ petals and core $Y^*$. We may repeat this procedure as long as $\ell$ satisfies (\ref{ell-bound}), i.e., until $\ell = p  r^{s-k-1}$ . This implies that the number of petals in sunflower $\mathcal{B}$ is at least 
\[
p  r^{s-k-1}.
\]
As $\mathcal{B}$ has core $Y^*$ of size $c(\mathcal{B}) < s$ we have a contradiction to the choice of sunflower in step $t$.
\end{proof}

\end{document}